\documentclass[reqno]{amsart}

\usepackage{mathrsfs}
\usepackage{amsmath}
\usepackage{amsthm}
\usepackage{amsfonts}
\usepackage{amssymb}
\usepackage{url}
\usepackage{enumerate}
\usepackage[pdftex,bookmarks=true]{hyperref}
  \usepackage[usenames, dvipsnames]{color}
\usepackage{verbatim}

\usepackage{tikz}
\usepackage{subcaption}
\usepackage{pgfplots}
\usepackage{adjustbox}

\usepackage{pdfsync}

\newcommand\rank{{\operatorname{rank}}}

\renewcommand\P{{\mathbf{P}}}
\newcommand\E{{\mathbf{E}}}

\newcommand\tr{{\operatorname{tr}}}

\newcommand\M{{\operatorname{M}}}

\newcommand\Z{{\mathbf{Z}}}

\newcommand\F{{\mathbf{F}}}

\newcommand\col{{\mathbf{c}}}
\newcommand\row{{\mathbf{r}}}

\newcommand\al{\alpha}

\newcommand\Ba{{\mathbf a}}

\newcommand\Be{{\mathbf e}}

\newcommand\Bv{{\mathbf v}}
\newcommand\Bw{{\mathbf w}}
\newcommand\Bx{{\mathbf x}}

\newcommand\BX{{\mathbf X}}

\renewcommand\Pr{{\mathbf P }}

\newcommand\CE{{\mathcal E}}
\newcommand\CF{{\mathcal F}}

\newcommand\supp{\mathbf{supp}}

\newcommand\eps{\varepsilon}

\newcommand\codim{\operatorname{codim}}
\newcommand\lang{\langle}
\newcommand\rang{\rangle}

\newcommand\bs{\backslash}

\newcommand\rk{{\operatorname{rank}}}

\newcommand\cork{{\operatorname{corank}}}

\newcommand\GL{\operatorname{GL}}

\newcommand\CB{{\mathcal B}}

\newcommand{\Mat}{\operatorname{Mat}}

\newcommand{\un}{\operatorname{uniform}}

\usepackage{fullpage}

\theoremstyle{plain}
  \newtheorem{theorem}[subsection]{Theorem}

  \newtheorem{prop}[subsection]{Proposition}
  \newtheorem{proposition}[subsection]{Proposition}
  
  \newtheorem{lemma}[subsection]{Lemma}
  \newtheorem{corollary}[subsection]{Corollary}

  \newtheorem{condition}{Condition}

  \newtheorem{claim}[subsection]{Claim}

\theoremstyle{definition}
  \newtheorem{definition}[subsection]{Definition}

\begin{document}

\title{Rank of Near Uniform Matrices}

\author{Jake Koenig}
\author{Hoi H. Nguyen}

\address{Department of Mathematics\\ The Ohio State University \\ 231 W 18th Ave \\ Columbus, OH 43210 USA}
\email{koenig.427@osu.edu}
\email{nguyen.1261@osu.edu}

\thanks{The authors are supported by NSF grant DMS-1752345.}

\maketitle
\begin{abstract} 
	A central question in random matrix theory is universality. When an emergent phenomena is observed from a large collection of chosen random variables it is natural to ask if this behavior is specific to the chosen random variable or if the behavior occurs for a larger class of random variables.

	The rank statistics of random matrices chosen uniformly from $\operatorname{Mat}(\mathbf{F}_q)$ over a finite field are well understood. The universality properties of these statistics are not yet fully understood however. Recently Wood \cite{W1} and Maples \cite{M1} considered a natural requirement where the random variables are not allowed to be too close to constant and they showed that the rank statistics match with the uniform model up to an error of type $e^{-cn}$. In this paper we explore a condition called near uniform, under which we are able to prove tighter bounds $q^{-cn}$ on the asymptotic convergence of the rank statistics.
	
	Our method is completely elementary, and allows for a small number of the entries to be deterministic, and for the entries to not be identically distributed so long as they are independent. More importantly, the method also extends to near uniform symmetric, alternating matrices.
	
	Our method also applies to two models of perturbations of random matrices sampled uniformly over $\operatorname{GL}_n(\mathbf{F}_q)$: subtracting the identity or taking a minor of a uniformly sampled invertible matrix. 

\end{abstract}

\section{Introduction} 
A central question in random matrix theory is that of universality. If instead of taking a Haar-uniform random matrix, we sample by some other methods, under what conditions will we observe similar statistics in some large scale matrix phenomena? While  there  have  been  many  results  addressing universality of random matrices in characteristic zero to study the spectral behavior of various models of random  matrices,  we  have  not  seen  much  in  the  literature  addressing  universality  behavior  in  the  finite fields  setting.   In  fact,  to  the  best  of  our  knowledge,  although  there  had  been results for matrices over finite fields such  as  \cite{Balakin, BKW, BM, Cooper, CRR, KK,  KL1, KL2, Kop, SW}, universality results only appeared very recently in  \cite{E, LMNg, M1,M2, NgP, NgW1,NgW2, W1,W2}. 

The rank statistics of a uniform random matrix over a finite field $\F_q$ are well studied. An exact formula for the rank of a uniform matrix in $\F_q$ are folklore and can be found in for instance \cite{Belsley}. A formula for the rank of a symmetric matrix was found by MacWilliams \cite{Mac} and for an alternating matrix by MacWilliams and Sloane \cite{MacS}. We will recover these formulas throughout our presentation under a more probabilistic viewpoint.

Consider random matrices with entries being iid copies of $\xi$ satisfying a so called min-entropy condition that $\P(\xi =k) \le 1-\al$ for all $k\in \F_q$. The main result of \cite{M1} by Maples (see also \cite{NgP} for corrections) and a special case of the recent result \cite{W2} by Wood showed that the rank distribution of such matrices is asymptotically equal to the distribution of the Haar distributed random matrices. Maples' approach relies on a very fine machinery (swapping method) originating from \cite{BVW, KKSz, TV}, which achieves an exponential error bound $e^{-c_\al n}$, while Wood's approach relies on counting surjections (moment method) which yields weaker error but the method has wider applications. 

In this paper we introduce a condition on the entries which can be seen as an interpolating between uniform and constant min-entropy.  Let $\xi\in\F_q$ be a random variable satisfying,
\begin{equation}\label{eqn:xi}
c_k= \P(\xi = k) \le C/q,  \forall  k\in \F_q,
\end{equation}
where $C\ge 1$ is a given constant. We refer to $\xi$ satisfying equation~\ref{eqn:xi} as near uniform. In a sense, it resembles distributions with bounded density in the characteristic zero case.

Under this stronger hypothesis we show that the error bound can be improved to $q^{-cn}$, by a completely elementary approach. Our other main contributions are: (1) we do not require the entries to be identically distributed, only that they each satisfy the near uniform condition, and we do not even need all the entries to be near uniform as  we can allow a small linear number of entries in each column to be arbitrary (see Theorem \ref{thm:main:1} and Theorem \ref{thm:rec}); (2) our method can be extended to symmetric and alternating matrices (see Theorem \ref{thm:sym} and Theorem \ref{thm:alt}) where we also allow many entries to be arbitrary, (3) furthermore, our method also yields interesting applications towards the uniform model itself, such as perturbation of matrices sampled uniformly over $\GL_n(\F_q)$ (see Theorem \ref{thm:GL:1} and Theorem \ref{thm:GL:2}), or the evolution of rank of a matrix evolved from an arbitrary matrix (see Theorem \ref{thm:eventualequiv}). All of these results seem to be new. In what follows we detail our main results.

\subsection{Random Matrices of Independent Entries} In this section, if not stated otherwise, we assume that $X_1,\dots, X_{n}$ are independent random vectors in $\F_q^n$ with independent and near uniform entries. Let $M_n$ be the random $n \times n$ matrix with the $X_i$ as column vectors. Denote the corank of the matrix $M_n$ by 
$$Q(M_n):=n - \rank(M_n).$$ 
It is known (see for instance \cite{Belsley}) that for the uniform model $M_{\un,n}$, for any $0\le k\le n$ we have
$$\P(Q(M_{\un, n}) =k) = \frac{1}{q^{k^2}} \frac{\prod_{i=1}^n (1-1/q^i) \prod_{i=k+1}^n (1-1/q^i)}{\prod_{i=1}^{n-k} (1-1/q^i) \prod_{i=1}^k (1-1/q^i)}.$$ 
 Let $Q_\infty$ be the ``limiting" random variable with 
$$\P(Q_\infty=k) = \frac{1}{q^{k^2}} \frac{\prod_{i=k+1}^\infty (1-1/q^i)}{ \prod_{i=1}^{k}(1-1/q^i)}.$$
The following was shown by Fulman and Goldstein \cite{FG} 
\begin{equation}\label{eqn:FG:1}
\frac{1}{8q^{n+1}} \le \|Q(M_{\un,n}) - Q_\infty\|_{TV} \le \frac{3}{q^{n+1}}.
\end{equation}
This is a striking result, but as their method compares $Q(M_{\un,n})$  and $Q_\infty$ directly, it does not seem to extend beyond the uniform model. One of the main results of this note is the following.

\begin{theorem}[rank distribution of independent entry square matrices]\label{thm:sq} There exist constants $c, C'$ depending on $C$ such that
$$\|Q(M_n) - Q_\infty\|_{TV} \le \Big(\frac{C'}{q}\Big)^{cn}.$$
\end{theorem}

Next, motivated by problems from algebraic combinatorics to enumerate invertible matrices with given constraints of zero entries (see for instance \cite{L} and the references therein), we consider perturbations in the following way. For each $1\le i\le n$, let $F_i\subset [n]$ be the set of coordinates of $X_i$ which are allowed to be any random variable (independent of the others). These sets $F_i$ satisfy the following.

\begin{condition}\label{cond:F} Let $\al$ be a sufficiently small constant. 
\begin{itemize}
\item $|F_i| < \al n$. 
\vskip .1in
\item Every index $i\in[n]$ is contained in at most $(1-12\alpha)n$ sets $F_j$ \footnote{There is nothing special about the constant 12, we use it for convenience, and without intention to optimize.}.
\end{itemize}
  \end{condition}
These conditions are necessary as our result does not hold if the matrix has (many) fixed rows or columns. One typical example is $F_i=\{i,\dots, i+\al n\}$ (which corresponds to the case that the entries of distance $\al n$ from the main diagonals are allowed to be deterministic, for instance all of them can be zero).  

\begin{definition}
We say that $X_i$ is a near uniform vector of {\it type} $F_i$ if the entries of $X_i$ are independent and near uniform, except those with indices in $F_i$, which may be arbitrary.
\end{definition}
 
\begin{theorem}\label{thm:main:1} Let $F_1,\dots, F_n$ be index sets satisfying Condition \ref{cond:F}. Assume that $X_1,\dots, X_n$ are independent random vectors, where $X_i$ is of type $F_i$. Then there exist constants $C',c$ depending on $C$ and $\al$ such that
	$$\|Q(M_n) - Q_\infty\|_{TV} \le \Big(\frac{C'}{q}\Big)^{cn}.$$
\end{theorem}

Note that Theorem~\ref{thm:sq} is a corollary of Theorem~\ref{thm:main:1} where one can simply take the sets $F_i$ to be empty.

\begin{corollary}
The number of matrices $M_n \in \Mat_n(\F_q)$ of rank $r$, where $r \ge (1-\al)n$, and where the entries in $F_1,\dots, F_n$ are all zero, is asymptotically
$$N_{F_1,\dots, F_n} = q^{n^2 - \sum_i |F_i|} \Big(Q_\infty(n-r) + O((C'/q)^{cn})\Big).$$ 
\end{corollary}

We remark that invertible matrices with vanishing diagonal are studied in \cite{L}. A precise formula for the number of $n\times (m+n)$ matrices of rank $n$ with vanishing diagonal is given in \cite[Proposition 2.2]{L}. Relatedly, we show that the above estimate also holds for more general rectangular matrices (stated here without perturbations for convenience). Let $X_1,\dots, X_{n+m}$ be random independent vectors with near uniform entries in $\F_q^n$, and let $M_{n \times (n+m)}$ be the random $n \times (n+m)$ matrix spanned by the column vectors.  It is known (again from \cite{Belsley}) that for the uniform model $M_{\un, n,m}$, for any $0\le k\le n$ we have
$$\P(Q(M_{\un,n,m}) =k) = \frac{1}{q^{k(m+k)}} \frac{\prod_{i=1}^{n+m} (1-1/q^i) \prod_{i=k+1}^n (1-1/q^i)}{\prod_{i=1}^{n-k} (1-1/q^i) \prod_{i=1}^{m+k} (1-1/q^i)}.$$
Let $Q_{m,\infty}$ be ``limiting" (as $n\to \infty$) random variable with
$$\P(Q_{m, \infty}=k) = \frac{1}{q^{k(m+k)}} \frac{\prod_{i=k+1}^\infty (1-1/q^i)}{ \prod_{i=1}^{m+k}(1-1/q^i)}.$$
Note that \cite{FG} Fulman and Goldstein showed that the two distributions above are very close,
\begin{equation}\label{eqn:FG:rec}
\frac{1}{8q^{n+m+1}} \le \|Q(M_{\un,n,m}) - Q_\infty\|_{TV} \le \frac{3}{q^{n+m+1}}.
\end{equation}
Here we show
\begin{theorem}[rank distribution of independent entry rectangular matrices]\label{thm:rec} For the near uniform model $M_{n \times (n+m)}$ there exist constants $c, C'$ depending on $C$ such that
$$\|Q(M_{n \times (n+m)}) - Q_{m,\infty}\|_{TV} \le \Big(\frac{C'}{q}\Big)^{c(m+n)}.$$
\end{theorem}

\subsection{Random symmetric matrices} Now we discuss our result for symmetric matrices. Let the entries $m_{ij}, i\ge j$ of $M_n$ be independent near uniform random variables. It is known (from \cite{C,Mac}) that for the uniform model $M_{\un}$, for any $0\le k\le n$ we have
$$\P(Q(M_{\un, n}) =k) = \frac{1}{q^{\binom{n+1}{2}}} \prod_{i=1}^{\lfloor (n-k)/2 \rfloor} \frac{q^{2i}}{q^{2i-1}} \prod_{i=0}^{n-k-1} (q^{n-i}-1).$$

Let $Q_{sym, \infty}$ be the random variable with 
$$\P(Q_{sym, \infty}=k)= \frac{\prod_{i=0}^\infty (1-q^{-2i-1})}{\prod_{l=1}^k (q^i-1)}.$$

In the uniform model of symmetric matrices Fulman and Goldstein \cite{FG} showed
\begin{align}\label{eqn:FG:sym}
&\frac{0.18}{q^{n+1}} \le \|Q(M_{\un,n}) - Q_\infty\|_{TV} \le \frac{2.25}{q^{n+1}} \mbox{ for $n$ even}\nonumber \\
 & \frac{0.18}{q^{n+2}} \le \|Q(M_{\un,n}) - Q_\infty\|_{TV} \le \frac{2}{q^{n+2}} \mbox{ for $n$ odd}.
\end{align}

We show that 
\begin{theorem} For the near uniform symmetric model $M_n$ 
$$\|Q(M_n) - Q_{sym, \infty}\|_{TV} \le \Big(\frac{C'}{q}\Big)^{cn}.$$
\end{theorem}

Next, as in the iid case, we can extend the result to perturbations. More precisely we will sample $M_n$, a random symmetric matrix of size $n$ in the following way. All the entries on and above its diagonal are near uniform and independent except for those falling in index sets $F_i\subseteq [n]$ which satisfy the following condition:
\begin{condition}\label{condition:F2}Let $\al$ be a sufficiently small constant. 
	\begin{itemize}
		\item $|F_i|< \al n$.
		\vskip .1in
		\item For symmetry, we also assume that for every $i,j\in[n]$, $F_j$ contains $i$ if and only if $F_i$ contains $j$. In particular each index is contained in at most $\al n$ sets.
	\end{itemize}
\end{condition}

\begin{theorem}[rank distribution for random symmetric matrices]\label{thm:sym} With $M_n$ as above, there exist constants $C',c$ such that
$$\|Q(M_n) - Q_{sym, \infty}\|_{TV} \le \Big(\frac{C'}{q}\Big)^{cn}.$$
\end{theorem}

\begin{corollary}
The number of matrices $M_n \in \Mat_n(\F_q)$ of rank $r$, where $r \ge (1-\al)n$, and where the entries in $F_1,\dots, F_n$ are all zero, is asymptotically
$$N_{F_1,\dots, F_n} = q^{\binom{n+1}{2} - \sum_i |F_i|} \Big(Q_{sym, \infty}(n-r) + O((C'/q)^{cn})\Big).$$
In particular, when $F_1=\{1\},\dots, F_n=\{n\}$ then the number of full rank symmetric matrices with zero entries on the diagonal is $q^{\binom{n}{2}} \big(Q_{sym, \infty}(0) + O(1/q^{cn})\big)$, which is asymptotically the number of full rank symmetric matrices of size $n-1$. (In fact these two quantities are the same, see \cite[Theorem 3.3]{L}.) Similarly one can also deduce an asymptotic formula for the number of invertible matrices of the first $k$ entries zero (and so on), see also \cite[Theorem 4.25]{L}.
\end{corollary}

\subsection{Random alternating matrices (skew symmetric matrices)} Here we assume $q$ is odd. For alternating matrices $M_n^T = -M_n$.  It is well-known that the rank is even and is implied for example by Proposition~\ref{alt:even}. It is known (from \cite{C1}) that for the uniform model $M_{\un,n}$, for any $0\le k\le n$ of the same parity with $n$ we have
$$\P(Q(M_{\un,n}) =k) = \frac{1}{q^{\binom{n}{2}}} \prod_{i=1}^{(n-k)/2} \frac{q^{2i-2}}{q^{2i}-1} \prod_{i=0}^{n-k-1} (q^{n-i}-1).$$

Let $Q_{alt,e}$ and $Q_{alt,o}$ be the limiting random variables given by 
\[
	\P(Q_{alt,e} = k):= \begin{cases}
\prod_{i=0}^\infty (1 -q^{-2i-1}) \frac{q^k}{\prod_{i=1}^k (q^i-1)}, & \mbox{$k$ is even}\\
 0, \mbox{ $k$ is odd}
\end{cases}
\]
and
\[
	\P(Q_{alt,o}=k):= \begin{cases}
\prod_{i=0}^\infty (1 -q^{-2i-1}) \frac{q^k}{\prod_{i=1}^k (q^i-1)}, & \mbox{$k$ is odd}\\
 0, \mbox{ $k$ is even}
\end{cases}
\]

Let $M_n$ be a random matrix with independent near uniform entries above the diagonal except for entries in index sets $F_i$ satisfying \ref{condition:F2}. The entries on the diagonal are zero and the entries below the diagonal are the negative of the entries above the diagonal.

\begin{theorem}[Rank distribution for random alternating matrices]\label{thm:alt} There exist constants $c, C'$ depending on $C$ such that
	$$\|(Q(M_n) - Q_{alt,e}\|_{TV} \le  \Big(\frac{C'}{q}\Big)^{cn} \mbox{ if $n$ is even}$$
and
	$$\|(Q(M_n)- Q_{alt,o}\|_{TV} \le  \Big(\frac{C'}{q}\Big)^{cn} \mbox{ if $n$ is odd}.$$

\end{theorem}

Note that for the uniform model of random alternating matrices Fulman and Goldstein \cite{FG} showed
\begin{align}\label{eqn:FG:alt}
&\frac{0.18}{q^{n+1}} \le \|Q(M_{\un,n}) - Q_\infty\|_{TV} \le \frac{1.5}{q^{n+1}} \mbox{ for $n$ even.}\nonumber \\
&\frac{0.37}{q^{n+1}} \le \|Q(M_{\un,n}) - Q_\infty\|_{TV} \le \frac{2.2}{q^{n+1}} \mbox{ for $n$ odd.}
\end{align}

\subsection{Perturbations of matrices sampled uniformly from $\GL_n(\F_q)$} In our last subsection we discuss a few variants of the uniform model, for which the rank statistics do not seem to be trivial. First we single out an interesting special case of Theorem \ref{thm:sym}
 and Theorem \ref{thm:alt}, for which we cannot find in the literature.
 
 \begin{theorem}\label{thm:eventualequiv}  Let  $M_{m_0}$ be an arbitrary matrix of size $m_0$ which is  symmetric (or alternating). Let $M_n(M_{m_0})$ be the $n\times n$ random matrix with $M_{m_0}$ as its upper left hand corner and the remaining entries above the diagonal sampled uniformly and those below chosen to make the matrix symmetric (or alternating). Then,
	\begin{align*}
		\|Q(M_n(M_{m_0})) - Q_\bullet\|_{TV} \le  \frac{3^{n/2}}{q^{n/2-m_0}} .
	\end{align*}
	Where $Q_\bullet$ is either $Q_{sym,\infty}$,$Q_{alt,e}$ or $Q_{alt,o}$ depending on whether $M_n$ is symmetric or alternating and the parity of $n$. 
\end{theorem}
We remark that the range of $m_0$ above is nearly optimal as the result no longer holds for $m_0 \ge (1+o(1))n/2$. We also note that indeed this result plays an important role in our deduction of Theorem \ref{thm:sym}
 and Theorem \ref{thm:alt}, see also Proposition \ref{prop:eventualequiv}. 

We next focus on a few other models that can be treated using our method. 

\begin{theorem}[Perturbation of $\GL_n$]\label{thm:GL:1} Let $A_n$ be chosen uniformly from $\GL_n(\F_q)$ and let $M_n= A_n-I$, then there exist constants $C,c$ such that
	$$\|Q(M_n) - Q_\infty\|_{TV} \le \Big(\frac{C}{q}\Big)^{cn}.$$
\end{theorem}
This model was first considered in \cite{Wa} by Washington in his study of the Cohen-Lenstra heuristic, where it was shown that $\lim_{n \to \infty} \P(Q(M_n)=r) = Q_\infty(r)$. Here we show that the rate of convergence is extremely fast.  We also remark that one might be able to study this model using the cycle index technique by Kung and Stong (see for instance \cite{Fulman-BAMS}), however it is not clear how far one can quantify the speed of convergence.

Finally, our method also gives the following rank statistics of corners of $\GL_n(\F_q)$, a model considered recently by Van Peski \cite{Peski} in his study of the singular numbers of products.

\begin{theorem}[Corner of $\GL_n$]\label{thm:GL:2} Let $\eps>0$ be given and let $A_n$ be chosen uniformly from $\GL_n(\F_q)$. Let $M_{n'}$ be the top left corner of size $n'$ of $A_n$, where $n'\le (1-\eps)n$. 
	$$\|Q(M_{n'}) - Q_{\infty}\|_{TV} \le \frac{3}{q^{n'}} + \frac{2^{\min(n', \eps n)+1}}{q^{\eps n}}.$$
\end{theorem}

Note there's nothing special about the upper left corner so the same result holds for any fixed $n'\times n'$ minor.

\subsection{Organization of Paper}
In Section~\ref{section:iid} we prove Theorems~\ref{thm:main:1} and \ref{thm:rec}. To do this we introduce some elementary anti-concentration probabilities and a notion of structure in subsection~\ref{subsection:concentration}. In that section we count the number of structured vectors in Lemma~\ref{lemma:counting} which is used to bound the probability that there is a structured normal vector to the column span in both the independent case, in Proposition~\ref{prop:W_m}, and the symmetric and alternating cases in Proposition~\ref{prop:W_m:sym}. It is shown in Lemma~\ref{unconcimpliesuniform} that the probability a near uniform vector is contained in a subspace with only unstructured normal vectors is similar to the probability a uniform random vector is contained in such a subspace.

For the independent and alternating models these observations suffice. But for the symmetric model the rank sometimes increases by one or two so we need to understand quadratic forms. This analysis uses the decoupling lemma and is done in Lemma~\ref{lemma:quad} of Section \ref{section:sym}.

Finally we will prove Theorem \ref{thm:GL:1} and Theorem \ref{thm:GL:2} in Section \ref{section:GL}.

\subsection{Notation}
We use $M_n=(M_n(ij))_{1\le i,j\le n}$ to denote a square matrix of size $n$ with entries $M_n(ij)$ from $\F_q$. This matrix will be either non-symmetric, symmetric, or alternating of near uniform entries or a perturbed invertible matrix depending on the section. 

We denote by $r_i(M_n), c_j(M_n)$ the $i^{th}$ row and $j^{th}$ column of $M_n$ respectively. When $M_n$ is one of our models of random matrix and $N$ is a fixed $k\times k$ matrix we let $M_n(N)$ denote a matrix with upper left corner $N$ and other entries distributed like those of $M_n$. 

We denote by $W_k(M_n)$ the column span of the first $k$ columns of $M_n$. We write simply $W_k$ when the matrix is clear.

Let $\Be_i$ denote the $i$-th basis vector $(0,\dots,0,1,0,\dots,0)$.

We write $\P$ for probability and $\E$ for expected value. Sometimes we write $\P_X(.)$ to emphasize that the probability is taken with respect to $X$. For an event $\mathcal{E}$, we write $\bar{\mathcal{E}}$ for its complement. 

For a given index set $J \subset [n]$ and a vector $X= (x_1,\dots, x_n)$, we write $X|_J$ or sometimes $X_J$ to be the subvector of $X$ of components indexed from $J$.  Similarly, if $H$ is a subspace then $H|_J$ or $H_J$ is the subspace spanned by $X|_J$ for $ X\in H$.  

For $W\subset V$ a subspace we write $\overline{W}$ for the complement of $W$ in $V$ and $W^\perp$ for the orthogonal complement.

Finally we let 
$$e_q(t) := \exp(2\pi i \ \tr(t)/p)$$ 
where $p$ is the characteristic of $\F_q$. Note that the trace is equally distributed in $\F_p$ so for every $a\in \F_q$ we have,
\begin{align}\label{eqn:trace}
	1_{a=0} = \frac{1}{q}\sum_{t\in \F_q}e_q(at).
\end{align}

\section{Independent Entry Random matrices: proof of Theorem \ref{thm:main:1} and Theorem \ref{thm:rec}}\label{section:iid}
We first note that throughout the paper the constants $c,C'$ in $(C'/q)^{-cn}$ might differ from case to case. 

Let $W_k$ be the subspace generated by $X_1,\dots, X_k$ and $M_k$ be the matrix with $X_1, \dots X_k$ as columns. The approach is a column exposure process illustrated by Figure~\ref{ColumnEP}.

\begin{figure}

	\centering

	\begin{tikzpicture}
		\draw [decorate,decoration={brace,amplitude=10pt},xshift=-4pt,yshift=0pt]
		(0.8,5) -- (4.5,5) node [black,midway,yshift=0.6cm] 
		{\footnotesize $M_k$};
		\draw (1,1) -- (1,4) node[above=0.1in]{$X_{1}$} ;
		\draw (2,1) -- (2,4) node[above=0.1in]{$X_{2}$} ;
		\node at (2.5,2.5)[circle,fill,inner sep=1pt]{};
		\node at (3,2.5)[circle,fill,inner sep=1pt]{};
		\node at (3.5,2.5)[circle,fill,inner sep=1pt]{};
		\draw (4,1)  -- (4,4) node[above=0.1in]{$X_{k}$}; 
		\draw (5,1)  -- (5,4) node[above=0.1in]{$X_{k+1}$}; 
	\end{tikzpicture}
	\caption{column exposure process}
	\label{ColumnEP}
\end{figure}
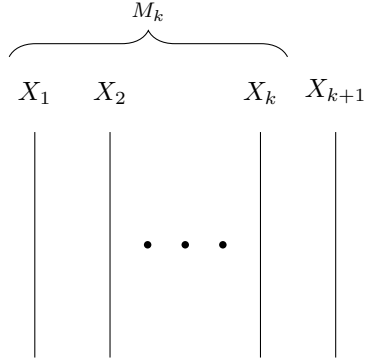

The idea is to show that the rank statistics evolve similarly to how they evolve in the uniform case when we expose the columns $X_1,\dots, X_n$ one by one. Note that in the uniform case we have
\begin{align}\label{eqn:uni:iid}
\P_{X_{k+1}}\Big(\rank(M_{k+1})=\rank(M_k)|W_k \wedge \rank(M_k)=l\Big) &= \P(X_{k+1} \in W_k|W_k \wedge \rank(M_k)=l)\nonumber \\
&= \frac{|W_k|}{q^n} = \frac{1}{q^{n-l}}.
\end{align}

To show that the rank for the near uniform model evolves asymptotically as in \eqref{eqn:uni:iid}, there are two regimes to understand. We first show that for $k\leq (1-\eps) n$, $M_k$ is full rank with high probability. This is Lemma~\ref{lemma:full}. For $k\geq (1-\eps)n$, $M_k$ is no longer full rank with high probability. We show the probability an additional column increases the rank matches with the uniform model by combining Proposition~\ref{prop:W_m} and Lemma~\ref{unconcimpliesuniform}. These pieces are combined to prove the main theorems at the end of the section.

\subsection{Full rank for thin matrices and non-sparsity of normal vectors}\label{subsection:non-sparse} 

We first show that if $k<(1-6\al)n$ then with high probability our vectors $X_i$ are independent.

\begin{lemma}[Odlyzko]\label{odlyzko} Let $V$ be a subspace of codimension $d$ in $\F_q^n$ and $X$ be a random vector with independent near uniform entries except for up to $k$ entries which may have arbitrary distribution. Then,
	\begin{align*}
		\P(X\in V)\leq \Big(\frac{C}{q}\Big)^{d-k}.
	\end{align*}
\end{lemma}
\begin{proof}
	Let $V = \langle v_1,\dots, v_{n-d}\rangle$ and consider the $n\times (n-d)$ matrix with the $v_i$ as columns. Without loss of generality suppose the first $n-d$ rows are linearly independent. Restricting to the first $n-d$ rows we have $X|_{[n-d]} = c_1v_1|_{[n-d]} + \dots + c_{n-d}v_{n-d}|_{[n-d]}$ for unique $c_i$.. If $X\in V$ then $X = c_1v_1 + \dots + c_{n-d}v_{n-d}$. This implies we have equality in the last $d$ coordinates of $X$. By assumption at least $d-k$ of the last $d$ entries of $X$ are independent and equal a fixed value with probability at most $C/q$. The result follows.
\end{proof}
	
\begin{lemma}\label{lemma:full} Let $\al$ be as in Condition \ref{cond:F}. With probability at least $1-n(C/q)^{5\al n}$, the matrix generated by $X_1,\dots, X_{\lfloor (1-6 \al)n\rfloor}$ has full rank.
\end{lemma}
\begin{proof} Recall that $X_i$ has type $F_i$ and $|F_i|< \al n$.
	If the vectors $X_1,\dots X_{\lfloor (1-6 \al)n\rfloor}$ do not have full rank then for some $i$ we have,
	\begin{align*}
		X_i\in \langle X_1\dots X_{i-1}\rangle \text{ and } X_1\dots X_{i-1}\text{ are linearly independent}.
	\end{align*}
	By Lemma~\ref{odlyzko} we can bound the probability of this event by,
	\begin{align*}
		\P(X_i\in \langle X_1\dots X_{i-1}\rangle | X_1\dots X_{i-1}\text{ are linearly independent})&\leq (C/q)^{n-i+1-n\al} \\ 
		&\leq (C/q)^{n - (\lfloor (1-6 \al)n\rfloor -n\al} \\
		&\leq (C/q)^{5\al n}.
	\end{align*}
	Taking the union bound over all $(1-6 \al)n$ places where the rank may drop we get the desired bound.
\end{proof}

Therefore, with a loss of at most $n(C/q)^{5\al n}$ in probability it suffices to assume that $W_{ (1-6\al)n}$ has full rank. 

Next we show that $W_k$ does not have a ``sparse" normal vector for $k\ge (1-6\al)n$ with high probability. This is important to the argument because we allow some entries of our matrix to be deterministic. If we had a normal vector $\Ba$ with support contained in one of our bad sets $|F_i|$ we wouldn't be able to estimate the probability $X_{k+1}\cdot \Ba=0$. Roughly speaking, these events are needed towards establishing any asymptotic form of \eqref{eqn:uni:iid} as the event $X_{k+1}\in W_k$ is equivalent with $X_{k+1} \cdot \Ba=0$ for all normal vectors $\Ba$ of $W_k$.

\begin{lemma}\label{lemma:sparse} Let $\al$ be as in Condition \ref{cond:F}. Then there exist constants $C',c$ such that the following hold. For any $(1-6 \al)n \le k \le n$, with probability at least $1-(C'/q)^{cn}$ the following holds: any normal vector $\Ba$ of $W_k$ has at least $\al n$ non-zero coordinates.
\end{lemma}

\begin{proof} Let $\Ba$ be a prospective normal vector with $S=\supp(\Ba)$, and $s=|S| \le \al n$. There are less than $2^n q^{\al n}$ such vectors. For each $1\le i\le k$, if we have that $\bar{F}_i \cap S \neq \emptyset$ (that is $S\not \subset F_i$), then by \eqref{eqn:xi}, it is clear that 
$$\P(X_i \cdot \Ba=0) \le C/q.$$ 
Letting $k'$ denote the number of such indices, by Condition \ref{cond:F} we see that 
$$k' \ge k- (1 -12 \al)n \ge 6\al n.$$
Hence the probability that there exists a sparse $\Ba$ serving as the normal vector to $W_k$ is bounded by
$$ \sum_{\Ba, |\supp(\Ba)| \le \al n} \P(X_i\cdot \Ba=0, 1\le i\le \lfloor 6 \al n \rfloor) \le 2^n q^{\al n} (C/q)^{\lfloor 6 \al n \rfloor} \le (C'/q)^{cn}.$$ 
\end{proof}
In the next step we show that with high probability the normal vectors are not only sparse but do not have structures (see Proposition \ref{prop:W_m}). Thanks to the near uniform assumption, our method is much simpler compared to those of \cite{LMNg, M1,NgP, NgW1, W2}.
 
\subsection{Anti-concentration probability}\label{subsection:concentration} We first introduce concentration inequalities for vectors with randomness from \eqref{eqn:xi}. Let $X=(x_1,\dots,x_n)$ be a random vector of type $F$. Let $c_k^{i} = \Pr(x_i=k)$ ,where by definition we know that $c_k^i\leq C/q$.

Let $X$ be a near uniform random vector of type $F$. By the Fourier transform for any $r\in \F_q$ and fixed $\Ba\in F_q^n$ we have
  \[
    \P(X \cdot \Ba =r) = 
    \E (1_{X \cdot 
    \Ba=r}) = 
    \frac{1}{q} \sum_{t \in \F_q} \prod_{i \notin F} \E e_p(\tr(x_i a_i t)) e_p(-\tr(r't))
  \]
  where $r'$ depends on $r$ and other deterministic entries of $X$.
  By the triangle inequality
  \begin{align}\label{eqn:rho}
    |\P(X \cdot \Ba =r) -\frac{1}{q}| & \le \frac{1}{q} \sum_{t \in \F_q, t\neq 0} \prod_{i\notin F} |\E e_p(\tr(x_i a_i t))| \nonumber \\
	  & =  \frac{1}{q} \sum_{t \in \F_q, t\neq 0} \prod_{i \notin F} |\sum_{k\in\F_q} c^{i}_k e_p(\tr(k a_i t)) | \nonumber \\
    & =: \frac{1}{q} \sum_{t\neq 0} \prod_{i\notin F} f_i(ta_i).
  \end{align}
 Where $f_i(t) =|\sum_{k\in\F_q} c^{i}_k e_p(\tr(k t))|$.  Motivated by this, for convenience we define the following,
$$\rho_F(\Ba):=\frac{1}{q} \sum_{t\neq 0} \prod_{i\notin F} f_i(ta_i).$$

Conceptually $\rho_F(\Ba)$ measures the ``structure" of $\Ba$. Vectors with smaller $\rho_F$ are ``less structured". The concept is a waypoint in the proof. At a high level our approach is to show that normal vectors are likely unstructured and unstructured vectors are orthogonal to near uniform vectors with probability close to uniform.

The goal of this subsection is to show the following proposition which states that with high probability the column space of $W_k$ has no structured normal vectors.

\begin{prop}\label{prop:W_m} Given $C$ and $\al$ where $\al$ is sufficiently small. There exist constants $C',c$ such that the following holds for $m\ge (1-6\al)n$: with probability at least $1- (C'/q)^{cn}$ with respect to $X_1,\dots, X_{m}$ any nonzero vector $\Ba=(a_1,\dots,a_n)$ orthogonal to each $X_i$ has, 
	$$ \rho_{F_{m+1}}(\Ba) \le (C'/q)^{cn}.$$
\end{prop}

To prove the above result, observe that if $a \neq 0$ then, noting that $\sum_{t\in \F_q} e_p(\tr(a t))=q 1_{a=0}$, we have
\begin{align*}
	\sum_{t\in \F_q} f_i(ta)^2 &=\sum_{t\in F_q}\Big|\sum_{k} c^{i}_k e_p(\tr(kta))\Big|^2\\ &=\sum_{t\in F_q}\Big[ \sum_k (c^{i}_k)^2  + \sum_{k \neq k'} c^{i}_k c^{i}_{k'} e_p(\tr((k-k')ta)) \Big]\\
	&= q\sum_k (c^{i}_k)^2 \le q\max_k c^{i}_k \sum_k c^{i}_k \le C.
\end{align*}

As such, for any $K>0$, let $T \subset \F_q$ be the set of $a$ where $|f_i(ta)| \ge Kq^{-1/2}$ then 
\begin{equation}\label{eqn:T}
|T| \le Cq/K^2.
\end{equation}
  
This implies the following claim.
\begin{claim}\label{claim:1} For any $\M$ if $\Ba=(a_1,\dots,a_n)$ is such that for any $t\neq 0$, it has at least $M$ indices $i\in [n]\bs F$ such that $ta_i \notin T$. Then we have 
	\begin{align}\label{eqn:bad}\rho_F(\Ba) = \frac{1}{q} \sum_t \prod_{i \notin F} f_i(ta_i) \le (K q^{-1/2})^M.\end{align}
\end{claim}

Thus towards Proposition \ref{prop:W_m}, it suffices to show the following and then choose suitable $M$ and $K$.

\begin{lemma}\label{lemma:str:iid} With probability  at least $1- q^{6\al n} (2C)^n  (C/K^2)^{n-M-|F|}$ with respect to $X_1,\dots, X_{m}$, any normal vector $\Ba$ of $W_{m}$ satisfies the condition in Claim \ref{claim:1}.
\end{lemma}

To prove Lemma~\ref{lemma:str:iid} we bound the number of vectors not satisfying equation~\eqref{eqn:bad}. Then we bound the probability each of those vectors is a normal vector to our matrix. Counting the bad set is also important in the symmetric section so we split it into the following lemma.
\begin{lemma}[Counting lemma]\label{lemma:counting} Let $\CB$ be the set of vectors in $\F_p^n$ where 
$$\frac{1}{q} \sum_t \prod_{i \notin F} f_i(ta_i) \ge (K q^{-1/2})^M.$$
Then
$$|\CB| \le q 2^n |T|^{n-M-|F|} q^{M+|F|} .$$ 
\end{lemma}

\begin{proof} By Claim~\ref{claim:1} for any $\Ba\in\CB$ there exists $t_0$ such that $t_0a_i \in T$ for at least $n-M-|F|$ indices $i$.  The other $M+|F|$ entries of $\Ba$ can be any element of $\F_q$. Thus
$$|\CB| \le q\times \binom{n}{M} |T|^{n-M-|F|} q^{M+|F|}.$$
\end{proof}

\begin{proof}(of Lemma \ref{lemma:str:iid}) We estimate the probability of the complement event. For a given $\Ba=(a_1,\dots, a_n)\in \CB, \Ba \neq 0$, we estimate the probability that it is a normal vector to $W_{m}$. By Lemma \ref{lemma:sparse}, we just need to focus on the event that $\Ba$ has at least $\al n$ non-zero entries. As such, by Condition \ref{cond:F}, for each $1\le i\le m$ we have
$$\P(X_i \cdot \Ba =0) \le C/q.$$
So the probability that $\Ba$ is a normal vector to $W_{m}$ is bounded by 
$$\prod_{i=1}^{m} \P(\Ba \cdot \BX_i=0)\le (C/q)^{m}.$$
Taking a union bound over $\CB$, we obtain an upper bound for the probability of the existence of structured normal vectors, using \eqref{eqn:T}
\begin{equation}\label{eqn:2}
q 2^n |T|^{n-M-|F|} q^{M+|F|}   \times (C/q)^{n-6 \al n} \le  q^{6\al n} (2C)^n (|T|/q)^{n-M-|F|} \le q^{6\al n} (2C)^n  (C/K^2)^{n-M-|F|}.
\end{equation}
\end{proof}

\begin{proof}(of Proposition \ref{prop:W_m})
Choose $M=\beta n, K= q^\beta$ with small $\beta<1/2$ and with even smaller $\al$. We obtain a bound $(C'/q)^{-cn}$ for the bound on probability in Equation~\eqref{eqn:2}, and for the bound on structure in Claim \ref{claim:1}. 
\end{proof}

\subsection{The rank statistics in the exposure process}\label{subsection:exposure}

The motivation to consider structure is if all vectors orthogonal to a given subspace have small $\rho$ then the probability it contains a near uniform random vector behaves like the probability it contains a uniform random vector. 
\begin{lemma}\label{unconcimpliesuniform}
	Let $H\subset  \F_q^n$ be a fixed subspace of codimension $d$ such that for every nonzero $w\in H^\perp$ $|\P(X\cdot w=0) - 1/q|<\delta$ and let $X$ be a near uniform random vector of type $F$ as above. Then
	\begin{align*}
		|\P(X\in H) - 1/q^d| \leq 2 \delta.
	\end{align*}
	Note when applying the lemma it suffices to bound $\rho_F(w)$ because $|\P(X\cdot w = 0) - 1/q|\leq \rho_F(w)$.
\end{lemma}
We also refer the reader to \cite{LMNg} for a variant for fields of prime order.
\begin{proof}
	We have the following identity,
	\begin{align*}
		1_{a_1 = 0 \wedge \dots \wedge a_d = 0} = \frac{1}{q^d}\sum_{t_1,\dots,t_d\in\F_q} e_q(a_1t_1 + \dots + a_dt_d).
	\end{align*}
	Therefore, letting $\Bv_1,\dots, \Bv_d$ be a basis for $H^\perp$,
	\begin{align*}
		\P(X\in H) &= \P(X\cdot \Bv_1 \wedge \dots \wedge X\cdot \Bv_d) \\
		&= \E\Big[ \frac{1}{q^d}\sum_{t_1,\dots,t_d\in\F_q} e_q(t_1[X\cdot \Bv_1] + \dots + t_d[X\cdot \Bv_d])\Big] \\
		&= \E\Big[ \frac{1}{q^d} +  \frac{1}{q^d}\sum_{\substack{t_1,\dots,t_d\in\F_q\\ \text{ not all $t_i$ are zero}}} e_q(t_1[X\cdot \Bv_1] + \dots + t_d[X\cdot \Bv_d])\Big]\\
		&=  \frac{1}{q^d} +  \frac{1}{q^d}\sum_{\substack{t_1,\dots,t_d\in\F_q\\ \text{ not all $t_i$ are zero}}} \E e_q(t_1[X\cdot \Bv_1] + \dots + t_d[X\cdot \Bv_d]).
	\end{align*}
	We now split the sum into projective equivalence classes. Let $\sim$ be the equivalence relation given by $(t_1,\dots,d_d)\sim(t_1',\dots,t_d')$ if there exists $t\neq 0$ such that $(t_1,\dots,t_d) = (t\cdot t_1',\dots,t\cdot t_d')$. Not worrying about our choice of representative on account of our inner sum we can continue,
	\begin{align*}
		&= \frac{1}{q^d} + \frac{1}{q^d}\sum_{(t_1,\dots,t_d)\in (\F_q^n)^\times/\sim}\bigg[\sum_{t\in\F_q^\times} e_q(t (t_1[X\cdot \Bv_1] + \dots + t_d[X\cdot \Bv_d]))\bigg] \\
		&= \frac{1}{q^d} + \frac{1}{q^d}\sum_{(t_1,\dots,t_d)\in (\F_q^n)^\times/\sim}q\bigg[\frac{1}{q}\sum_{t\in\F_q^\times} e_q(t(X\cdot \sum_{i=1}^n t_i \Bv_i))\bigg].
	\end{align*}

	Now observe,
	\begin{align*}
		\bigg[\frac{1}{q}\sum_{t\in\F_q^\times} e_q(t(X\cdot \sum_{i=1}^n t_i \Bv_i))\bigg] = \P(X\cdot \sum_{i=1}^n t_i \Bv_i = 0) - 1/q.
	\end{align*}
	
	The $\Bv_i$ are a basis and the $t_i$ are not all zero so $\sum_{i=1}^n t_i\Bv_i$ is a nonzero element of $H^\perp$. By assumption then this is bounded by $\delta$. There are $\frac{q^d-1}{q-1}$ elements in $(t_1,\dots,t_d)\in (\F_q^n)^\times/\sim$. We have,
	\begin{align*}
		\frac{q(q^d-1)}{q^d(q-1)} \leq 2.
	\end{align*}

	So by the triangle inequality we obtain,
	\begin{align*}
		|\P(X\in H) - 1/q^d| \leq 2 \delta.
	\end{align*}

\end{proof}

Let $\CE_m$ be the event considered in Proposition \ref{prop:W_m}, namely the event that all normal vectors $\Ba$ have $\rho(\Ba)<(C'/q)^{cn}$. By the proposition we know that 
$$\P(\CE_m) \ge 1-(C'/q)^{cn}.$$
\begin{proposition}\label{prop:evolution}  Assume that $(1 -6\al) n \le m\le n$ and the non-symmetric matrix $M_n$ is as in Theorem \ref{thm:sq}.  For each $l\le m$ we have 
$$\Big|\P(X_{m+1} \in W_m | \rk(W_m)=l \wedge \CE_m) - (1/q)^{n-l}\Big| \le (C'/q)^{cn}.$$
\end{proposition}
In other words this result confirm that our rank evolution matches with Equation \eqref{eqn:uni:iid} of the uniform model.

\begin{proof} By Proposition \ref{prop:W_m} on $\CE_m$, for any $\Ba \perp W_n$, we have $\rho_{F_{m+1}} (\Ba) \le (1/q)^{cn}$. Therefore by Lemma~\ref{unconcimpliesuniform},
$$|\P(X_{m+1} \in W_m | \rk(W_m)=l  \wedge \CE_m) -(1/q)^{n-l}| \le (C'/q)^{cn}.$$
\end{proof}
Now we prove our main results for non-symmetric matrices.

\begin{proof}(of Theorem \ref{thm:main:1}) We first condition on the events in Lemma \ref{lemma:full}, Lemma \ref{lemma:sparse}, and on $\wedge_{m\ge \lfloor (1-6\al)n \rfloor} \CE_m$ from Proposition \ref{prop:W_m}. For each $k \ge \lfloor (1-6\al)n \rfloor =:m_0$, the event $Q(M_n)=n-k$ can be written as follows using the column exposure process
	\begin{align*}
		\P(\rank(M_n)=k) &= \sum_{0< i_1<\dots <i_{k-m_0} \le n-m_0} \P\big( \wedge_j X_{m_0+i_j} \notin W_{m_0+i_j-1} \wedge \mbox{no rank increase at other places}\big).
	\end{align*}
	Now using Proposition~\ref{prop:evolution} we can estimate the RHS above as 
	\begin{align*}
		& = \sum_{0<i_1<\dots<i_{k-m_0}<n-m_0} \P(\text{a uniform matrix of size $n-m_0$ has rank drops at $i_1,\dots, i_{k-m_0}$}) + O(n (C'/q)^{cn})\\
		&= \P(\text{a uniform matrix of size $n-m_0$ has rank $k$}) + O(n (C'/q)^{cn}) \\
		&= Q_\infty(n-k) + O\bigg(1/q^{n-m_0} +n (C'/q)^{cn}) \bigg),
	\end{align*}
where we used \eqref{eqn:FG:1} in the last estimate.
\end{proof}

\begin{proof}(of Theorem \ref{thm:rec})
	For this we observe that the column rank is equal to the row rank and consider the $(n+m)\times n$ transpose of $M$. Our proof of Theorem~\ref{thm:main:1} shows that as we expose the columns the rank statistics match up to an error of type $(C'/q)^{c(m+n)}$ at each step.
\end{proof}

\section{Random symmetric and alternating matrices: proof of Theorem \ref{thm:sym} and \ref{thm:alt}}\label{section:sym}
In this section, if we don't specify otherwise, the arguments will work for both $M_n$ symmetric or alternating simultaneously. We let $M_m$ be the upper left minor of size $m$ of $M_n$. (Hence the notation is different from the non-symmetric case.)

Some of our ideas here are motivated by \cite{CTV, Mac,M3, Ng,V}, and in particular \cite{M3} \footnote{The main result of this paper has been improved substantially in \cite{NgW2}.}. The basic idea of our approach in the symmetric and alternating cases is similar to our approach in the independent case. But instead of considering a column exposure process, to preserve the independence between what we have and haven't seen, we expose one row and column at a time. In other words at the $m^{th}$ step we will have observed the upper left $M_m$ block and we'll be asking for the effect on the corank of expanding this block by one in both directions. 

Throughout this section we will use $X=X_{m+1}$ to refer to the top $m$ cells of column $m+1$ when we are at step $m$, that is $X=\col_{m+1}(M)|_{[m]}$. We also let $x_{m+1}$ be the entry $M_n(m+1,m+1)$, and $H_m$ be the column span of the exposed block at step $m$. This is summarized by Figure~\ref{CornerEP}.

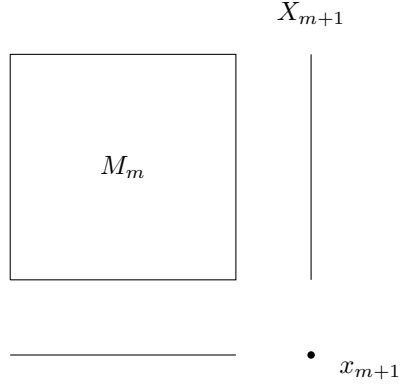
\begin{figure}
	\centering
	\begin{tikzpicture}

		\draw (1,1) -- (4,1) -- (4,4) -- (1,4) -- (1,1);
		\node at (2.5,2.5) {$M_m$};
		\draw (1,0)  -- (4,0);
		\node at (5,0)[circle,fill,inner sep=1pt]{};
		\draw (5,-0.2) node[right=0.1in]{$x_{m+1}$}; 
		\draw (5,1) --(5,4);

		\draw (5,4) node[above=0.1in]{$X_{m+1}$}; 
	\end{tikzpicture}
	\caption{corner exposure process}
	\label{CornerEP}
\end{figure}

The principle difference between the models is that if $X\in H_m$ then in the alternating model $M_{m+1}$ has the same rank as $M_m$. This is an elementary fact shown in Proposition~\ref{alt:even}. In the symmetric model when $X\in H_m$ the rank typically increases by $1$, but it depends on $x_{m+1}$. In Lemma~\ref{lemma:quad} we show the probability the rank increases by $1$ is what we'd predict from the uniform model.

\subsection{Almost full rank for $M_m$ and non-sparsity of generalized normal vectors}\label{subsection:non-sparse:sym}

This subsection is similar to Subsection \ref{subsection:non-sparse} where we show the following Odlyzko's lemma for random symmetric and alternating matrices, here we assume 
$$\sqrt{\al} n \le m \le n$$ 
where we recall that $\al$ is sufficiently small.  
 
\begin{lemma}\label{lemma:O:sym} Let $\eps$ be positive constant that is small but larger than $\sqrt{\al}$. We have
\begin{enumerate} 
	\item $\rank(M_m) \ge (1-\eps)m$ with probability  $1- m(C/q)^{(\eps -\sqrt{\al})m}$. 
\vskip .05in
\item\label{(3)} (non-sparsity of almost normal vectors) The following holds with probability $p$ at least $p\ge 1- (C'/q)^{(1-3\eps - \sqrt{\al})m}$: for any nonzero vector that is orthogonal to any of $(1-\eps)m$ columns of $M_m$  must have at least $\eps m$ non-zero components.
\vskip .05in
\item\label{iii}  (non-sparsity of solution vectors)  Let $Y_0$ be a fixed vector in $\F_q^m$. The following holds with probability at least $1- (C'/q)^{(1-3\eps - \sqrt{\al})m}$: any nonzero vector $\Bv$ for which $M_m \Bv$ agrees with $Y_0$ in at least $(1-\eps)m$ coordinates must have at least $\eps m$ non-zero components.
\end{enumerate}
\end{lemma} 

We will denote the intersection of these events by $\CF_m$. Also, for convenience, if $\Bv$ satisfies \eqref{iii} then we say that $\Bv$ is a {\it generalized normal vector} of $M_m$ (with parameter $\eps$). 
\begin{proof} 
	\begin{enumerate}[(1)]
		\item Let $k\leq m$ be an intermediate step. If $\cork(M_k)\geq r$ then by Lemma~\ref{odlyzko} (where we recall that the number of deterministic components is at most $\al n \le \sqrt{\al} m$) with probability at least $1-(C/q)^{(r-\sqrt{\al} m)}$ we have $X_{k+1}$ is not in the column span of $M_k$ which implies $\rk (M_{k+1}) = \rk(M_k) + 2$.
			
	If $M_m$ has corank greater than $\eps m$ there must be an intermediate step $M_k$ where the corank is greater than or equal to $\eps m$ but the rank doesn't increase by $2$. Taking a union bound over the possible indices we get,
	\begin{align*}
		\P(\cork(M_m)\geq \eps m)\leq m\Big(\frac{C}{q}\Big)^{(\eps -\sqrt{\al}) m}.
	\end{align*}
\item Observe that this case is a special case of case \ref{iii} where $Y_0 = 0$. 
	\item The number of candidate sparse vectors is,
		\begin{align*}
			{m\choose \eps m} q^{\eps m}.
		\end{align*}
			Note that the probability such a vector dotted with row $i$ is equal to any fixed value is bounded above by $\frac{C}{q}$ provided that $F_i$ does not contain the support of the vector. By Condition~\ref{condition:F2} at most $\sqrt{\al} m$ rows do. We don't have independence among all $\Bv\cdot X_i$ but we do have independence between $\Bv\cdot X_i$ and the other columns if $i$ is not in the support of $\Bv$. So because the support is assumed to be at most $\eps m$ the probability of equality is bounded above by,
			\begin{align*}
				{m\choose \eps m} q^{\eps m}\Big(\frac{C}{q}\Big)^{(1-2\eps - \sqrt{\al})m} \leq\Big(\frac{C'}{q}\Big)^{(1-3\eps - \sqrt{\al})m}.
			\end{align*}
			Provided $\eps,\al$ are sufficiently small.
	\end{enumerate}
\end{proof}

\subsection{Anti-concentration probability} \label{subsection:concentration:sym} Here and later we continue to assume that $\sqrt{\al} n \le m \le n$.
Our next step is similar to Subsection \ref{subsection:concentration} where we will show that generalized normal vectors do not have structures with very high probability.  We next recall from Equation~\eqref{eqn:rho} that $ \rho_F(\Ba)= \frac{1}{q} \sum_{t\neq 0} \prod_{i\notin F} f(ta_i)$, where $f(y) = |\sum_{k} c_k e_q(k y)|$, and
 $$|\sup_a \P(X \cdot \Ba = a) - 1/q| \le \rho_F(\Ba).$$ 

There is the following analog of Proposition \ref{prop:W_m} for generalized normal vectors. Loosely, it states that large symmetric near uniform matrices are unlikely to have structured vectors orthogonal to their row span. By symmetry the same result applies to the column span.
\begin{prop}\label{prop:W_m:sym} Assume that $\al$ is sufficiently small.
	There exist constants $c_1,c_2,C'$ such that with probability $1-(C'/q)^{c_1m}$ every nonzero generalized normal vector $\Ba$ of $M_m$ of parameter $\eps=2\sqrt{\al}$ has
	\begin{align*}
		\rho_F(\Ba) \leq (C'/q)^{c_2 m}.
	\end{align*}
\end{prop}
\begin{proof}
	Let $c_2 = (1/2 - \beta)\beta$ where $\beta$ is the constant in the proof of Proposition \ref{prop:W_m} (chosen appropriately, for instance $\beta=8\sqrt{\al}$ would work.) Let $\CB$ denote the set of $\Bv\in\F_q^m$ with $\rho_F(\Bv) > (C'/q)^{c_2 m}$. From Lemma~\ref{lemma:counting} we know,
	\begin{align*}
		|\CB| < (2C')^m q^{m-\beta m/2}.
	\end{align*}
	Without loss of generality we assume that the first $(1-\eps)m$ entries of $M_m\Ba$ are zero, and our aim is to bound this event,
	\begin{align*}
		\sum_{\Ba\in\CB\setminus\{0\}}\P\Big((M_m\Ba)_i = 0, 1\le i\le (1-\eps)m\Big).
	\end{align*}
	
	For all $\Ba\in\CB\setminus\{0\}$ we have the simple bound,
	\begin{align}\label{eq:orth:sym}
		\P\Big((M_m\Ba)_i = 0, 1\le i\le (1-\eps)m\Big) \leq (C/q)^{(1-\eps)m-1-\sqrt{\al} m}.
	\end{align}

	Because $\Ba\neq 0$ it has some entry $a_i\neq 0$. Conditioning on all the entries of $M_m$ outside of the $i^{th}$ column and row the probability that the $j^{th}$ row of $M_m$ has $r_j(M_m)\cdot Z = 0$, provided $j\notin F_i$, is bounded above by $C/q$ because of the near uniformity of the entry $M_m(j,i)$. The $M_m(j,i)$ for $j$ fixed and not equal to $i$ are independent so we get inequality~\eqref{eq:orth:sym}.

	So  we obtain,
	\begin{align*}
		\sum_{\Ba\in\CB\setminus\{0\}}\P(\Ba \mbox{ is a generalized normal vector of $M_m$} ) &\leq \binom{m}{\lfloor (1-\eps) m\rfloor}(2C')^mq^{m-\beta m/2}(C/q)^{(1-\eps)m-1-\sqrt{\al} m} \\
		&\leq (C''/q)^{-(\beta/2- 3\sqrt{\al}) m}.
	\end{align*}
\end{proof}

\subsection{The rank statistics in the exposure process}\label{subsection:exposure:sym} 
The analysis to understand the rank evolution in the corner exposure process is more involved than in the independent case. The technique hinges on the following so called Decoupling Lemma. The technique was first used to analyze symmetric random matrices by Costello, Tao and Vu \cite{CTV} though the idea is old and common in number theory.

\begin{lemma}[Decoupling]\label{lemma:decoupling1} Assume that $a_{ij} \in \F_q$ and $a_{ij}=a_{ji}$ and $b_i \in \F_q$. Assume that $x_i$ are independent random variables. Then for any index set $I \subset [m]$ we have 
$$\sup_r|\P(\sum_{ij} a_{ij} x_i x_j + \sum_i b_i x_i = r) - 1/q|^4 \le |\P(\sum_{i\in I, j \in I^c} a_{ij} y_i y_j=0) -1/q|,$$
where $y_i\equiv x_i - x_i'$ with $x_i'$ an iid copy of $x_i$.
\end{lemma}

\begin{proof} For short we write $f(X)=\sum_{ij} a_{ij} x_i x_j + \sum_i b_i x_i$. We write
$$|\P(f(X)= 0) -\frac{1}{q}| = |\frac{1}{q}  \sum_{t\neq 0} \E e_q(-tf(X))|.$$
We then use Cauchy-Schwarz to complete squares, 
\begin{align}\label{eqn:dcpl}
LHS^2 \le  \frac{q-1}{q^2} \sum_{t\neq 0} |\E  e_q(-tf(X))|^2 &\le \frac{q-1}{q^2} \sum_{t\neq 0} \E_{X_I} |\E_{X_{I^c}}  e_q(-tf(X_I,X_{I^c})) |^2  \\
&= \frac{q-1}{q^2} \sum_{t\neq 0} \E_{X_I} \E_{X_{I^c}, X'_{I^c}}  e_q(-t[f(X_I,X_{I^c})-f(X_I,X_{I^c}')]) \nonumber \\
&= \frac{q-1}{q^2} \sum_{t\neq 0}  \E_{X_{I^c}, X'_{I^c}} \E_{X_I} e_q(-t[f(X_I,X_{I^c})-f(X_I,X_{I^c}')]) \nonumber.
\end{align}
Using Cauchy-Schwarz once more, we obtain
\begin{align*}
LHS^4 & \le (\frac{q-1}{q^2})^2 (q-1)\sum_{t\neq 0} \E_{X_{I^c}, X'_{I^c}} \E_{X_I, X_I'} e_q(-t[f(X_I,X_{I^c})-f(X_I,X_{I^c}') -f(X_I',X_{I^c})+f(X_I',X_{I^c}')  ]) \\
&= (\frac{q-1}{q^2})^2 q(q-1) \frac{1}{q} \sum_{t \neq 0} \E_{Y_I,Y_{I_c}} e_q(-t \sum_{i \in I, j\in I^c} a_{ij} y_i y_j)\\
& = (\frac{q-1}{q})^3 (\P(\sum_{i\in I, j \in I^c} a_{ij} x_i y_j =0) -1/q).
\end{align*}
\end{proof}

By Lemma \ref{lemma:O:sym}, with probability at least $1-(C'/q)^{cm}$ we can assume that $M_m$ has rank at least $(1-\eps)m$. Let us consider the event that $M_m$  has rank $m-k$ (where $k\le \eps m$).

\begin{claim}\label{claim:principleminor} Assume that $G_m$  has rank $m-k$, then there is a set $I \subset [m], |I| =m-k$ such that the principle minor matrix $G_{I \times I}$ has full rank $m-k$.
\end{claim}
\begin{proof} If suffices to consider the symmetric case. Assume without loss of generality that $\row_1(M_m), \dots, \row_{m-k}(M_m)$ span the row vectors of $G_m$, then in particular $\row_i|_{[m-k]}, i\ge n-k+1$ belong to the span of  $\row_i|_{[m-k]}, 1\le i\le m-k$. This implies that the matrix spanned by the first $m-k$ columns $X_1,\dots, X_{m-k}$ has rank at most $m-k$. On the other hand, by the symmetry of the matrix (for both symmetric and alternating), this column matrix has the same rank as that of the matrix generated by $\row_1(M_m), \dots, \row_{m-k}(M_m)$, which is $m-k$. So the matrix $G_{[m-k] \times [m-k]}$ generated by $\row_i|_{[m-k]}, 1\le i\le m-k$ has rank $m-k$. 
\end{proof}

Therefore in what follows, assume without loss of generality that $\row_1(M_m), \dots, \row_{m-k}(M_m)$ span the row space of $M_m$. 

{\underline{Probability for $\rk(M_{m+1})\le \rk(M_m)+1$}:} When we add the new column $X_{m+1} = (x_{1(m+1)},\dots, x_{(m+1) (m+1)}):=(x_1,\dots, x_{m+1})$ and its transpose (or negative transpose in the alternating case), the event $\rk(M_{m+1})< \rk(M_m)+2$ is equivalent with the event that the extended row vector $\row_1(M_{m+1}),\dots, \row_{m-k}(G_{m+1})$ still generate the space of the vectors $\row_1(M_{m+1}),\dots, \row_m(M_{m+1})$. In particular, this hold if for $m-k+1\le i\le m$
\begin{equation}\label{eqn:x,a}
x_{i} = \sum_{j=1}^{n-k}  a_{ij} x_{j},
\end{equation}
where $a_{ij}$ are determined from $M_m$ via
$$\row_i(M_m) = \sum_{j=1}^{m-k} a_{ij} \row_j(M_m).$$
In other words, equation~\eqref{eqn:x,a} says that the vector $(x_{1}, \dots, x_{m})$ is orthogonal to the vectors $(a_{i1},\dots, a_{i(m-k)}, 0,\dots,-1,0,\dots,0)$, or equivalently it belongs to the hyperplane $H_m$ generated by the column vectors of $M_m$. We next pause to record the evolution of the uniform model (where $(x_{1}, \dots, x_{m})$ is chosen uniformly from $\F_q^m$):  assume that $k\ge 1$. Then (see also \cite[Lemma 4]{Mac})
\begin{equation}\label{ean:uni:sym}
\P_{X_{m+1}}\big(\rk(M_{m+1})\le \rk(M_m)+1|\rank(M_m)=m-k\big) =\frac{1}{q^k}.
\end{equation}
Now, let $\CE_{m}$ be the event from Proposition \ref{prop:W_m:sym}. Then we have $\P(\CE_m)=1-(C'/q)^{-cm}$. By this proposition, and by Lemma \ref{unconcimpliesuniform}, we have

\begin{lemma}\label{lemma:linear} Assume that $k\ge 1$. Then
$$\Big|\P_{X_{m+1}}\big(\rk(M_{m+1})\le \rk(M_m)+1|\rank(M_m)=m-k \wedge \CE_{m} \big) - \frac{1}{q^k} \Big|$$
$$= \Big|\P\big(X_{m+1} \in H_m | \rank(M_m)=m-k \wedge \CE_{m}\big) -\frac{1}{q^k}  |\le (C'/q)^{cn}.$$
\end{lemma}

{\underline{Probability for $\rk(M_{m+1})= \rk(M_m)$}:} For alternating matrices we have $X_{m+1}\in H_m$ implies $Q(M_{m+1}) = Q(M_m)+1$ by the following elementary proposition.

\begin{proposition}\label{alt:even} For $A$ an alternating $m\times m$ matrix and $\Bx$ a column vector of length $m$ the following block matrix has either corank $Q(A)-1$ or $Q(A)+1$,
	\begin{align*}
	\begin{pmatrix}
		A & \Bx \\
		-\Bx^T &  0
	\end{pmatrix}.
	\end{align*}
\end{proposition}
\begin{proof}
	If $\Bx$ is not in the column span of $A$ then $[-\Bx^T\ 0]$ is not in the row span of $[A\ \Bx]$ so the rank increases by $2$.

	For the other direction, if $\Bx$ is in the column span of $A$, so $A\Ba = \Bx$ for some $\Ba$ we need to show that $-\Bx^T \Ba = 0$ so that the addition of the row will not increase the rank. Because $A$ is alternating we have $A = -A^T$. Therefore $(\Ba^TA\Ba)^T = \Ba^TA^T\Ba = -\Ba^TA\Ba$. This implies $\Ba^TA\Ba = 0$ and $\Ba^T\Bx = 0$.
\end{proof}

For symmetric matrices, $\rank(M_{m+1}) \le \rk(M_m)+1$ does not automatically imply that $\rank(M_{m+1}) =\rk(M_m)$, so we have to consider the events $\rank(M_{m+1}) =\rk(M_m)$ and $\rank(M_{m+1}) =\rk(M_m)+1$ separately. Let's focus on the event that 
$$\rank(M_{m+1}) =\rk(M_m).$$ 
Here besides the event considered in Lemma \ref{lemma:linear}, $\row_{m+1}(M_{m+1})$ also belongs to the subspace generated by $\row_1(M_{m+1}),\dots, \row_m(M_{m+1})$.  Knowing that $M_m$ has rank $m-k$, by Claim \ref{claim:principleminor} we can assume that $M_{I \times I}$ is a submatrix of full rank $m-k$ in $M_m$, for some $I\subset [n]$ and $|I|=m-k$. Let $B=(b_{ij})$ be the inverse of $M_{I \times I}$ in $\F_q$. The exposure of $X=(x_1,\dots, x_{m-k}, x_{m+1})$ would then increase the rank of $M_{m+1}$ by one, except when
\begin{equation}\label{eqn:b_{ij}}
\sum_{ij} b_{ij} x_i x_j +x_{m+1} = 0.
\end{equation}
For short we write $f(X_{m+1}):= \sum_{1\le i,j\le m-k} b_{ij} x_i x_j + x_{m+1}$. We next pause to record the evolution of the uniform model (where $(x_{1}, \dots, x_{m+1})$ is chosen uniformly from $\F_q^m$). The probability $\P(X \in H_m \wedge f(X_{m+1})=0)$ can be simply reduced to $\P(X \in H_m) \times \P_{x_{m+1}} (f(X_{m+1}=0|x_1,\dots, x_m)  =\frac{1}{q^k} \times \frac{1}{q}$ because  $x_{m+1}$ is uniform and independent from the other entries. Hence we have in the uniform case (see also \cite[Lemma 4]{Mac})
\begin{equation}\label{ean:uni:sym'}
\P_{X_{m+1}}\Big(\rank(M_{m+1}) =\rk(M_m) | \rk(M_m)=n-k\Big)  =\frac{1}{q^{k+1}}.
\end{equation}

For our case, we cannot rely on $x_{m+1}$ because this entry can be deterministic. Furthermore, the random entries $x_1,\dots,x_m$ are not uniform either. Nevertheless, we show that the evolution is still asymptotically the same as in the uniform case, under the events $\CE_m \wedge \CF_m$ from Lemma \ref{lemma:O:sym} and Proposition \ref{prop:W_m:sym}. 
\begin{lemma}\label{lemma:quad} Let $k\ge 0$. We have 
	$$\Big|\P_{X_{m+1}}\big(\rank(M_{m+1}) =\rk(M_m) | \rk(M_m)=n-k \wedge \CE_m \wedge \CF_m \big) - q^{-k-1}\Big|$$ 
	$$ =\Big|\P_{X_{m+1}}\big(X \in H_m \wedge f(X_{m+1}) =0| \CE_m \wedge \CF_m\big) - q^{-k-1}\Big| \le (C'/q)^{cm},$$
	where $B = \{b_{ij}\}_{1\leq i,j,\leq m-k}$ is the inverse to the full rank $I\times I$ minor of $M_m$.
\end{lemma}

\begin{proof} The probability can be rewritten in the following way.
\begin{align*}
\P(X \in H_m \wedge f(X) =0) &= q^{-k-1} \sum_{\xi \in H_n^\perp, t \in \F_q} \E e_q(X \cdot \xi + f(X)t)\\
&=q^{-k-1} + q^{-k-1}\sum_{\xi \in H_m^\perp, t \in \F_q, t\neq 0 }  \E e_q(X \cdot \xi + f(X)t) +  q^{-k-1} \sum_{\xi \neq 0, \xi \in H_m^\perp} \E e_q(X \cdot \xi).
\end{align*}
Note that the third sum is 
$$q^{-k-1} \sum_{\xi \neq 0, \xi \in H_m^\perp} \E e_q(X \cdot \xi) = q^{-k-1} (\sum_{\xi \in H_m^\perp} \E e_q(X \cdot \xi) - 1) =\frac{1}{q}\P(X \in H_n) - q^{-k-1}.$$
Hence the third sum is bound by $(C'/q)^{cm}$ using Lemma \ref{lemma:linear}. 

	For the second summand we can apply Lemma~\ref{lemma:decoupling1} (more precisely Equation \eqref{eqn:dcpl}). With the $X,b_i,x_{m+1}$ in the context of this proof serving the role of $b_i,a_i,-r$ in the context of the lemma respectively. We obtain,
	$$\Big|q^{-k-1} \sum_{\xi \in H_m^\perp, t \in \F_q, t\neq 0} \E e_q(X \cdot \xi + f(X)t)\Big|^{1/4} \le \Big|\P(Y_{I_1} \cdot B Y_{I_2} = 0) - \frac{1}{q}\Big|. $$
	Where $I_1\cup I_2 = I$ and $Y_{I_1},Y_{I_2}$ are random vectors in $\F_q^{I_1},\F_q^{I_2}$ with entries $y_i$ distributed by $x_i-x_i'$. Let $Y_1',Y_2'$ denote random vectors in $\F_q^m$ which restrict to $Y_1,Y_2$ on $I_1,I_2$ and are zero elsewhere. Similarly let $B'$ be the $m\times m$ matrix with $I\times I$ minor equal to $B$ and zeros elsewhere.
Then $Y_{I_1} \cdot B Y_{I_2} = 0$ is equivalent to
$$Y_{I_1}' \cdot B' Y_{I_2}' =0.$$
	Now choose $|I_1|=\lfloor (1-\eps)m \rfloor$ and $I_2=I\backslash I_1$, where $\eps=2\sqrt{\al}$. Therefore the vector $Y_{I_2}'$ is non-zero with probability at least $1-(C'/q)^{\eps m}$ and that
$$M_m(B' Y_{I_2}') = (M_m B') Y_{I_2}'= Y_{I_2}'$$
It follows by Lemma \ref{prop:W_m:sym} that, as $B'Y_{I_2'}$ is a generalized normal vector of $M_m$ on $\CE_m$, and as $k+|I_2|\le \eps n$, we have 
$$\rho(B' Y_{I_2}'|_{I_1}) \le (C'/q)^{cm}.$$
It thus follows by the definition of $\rho$ that 
$$|\P_{Y_{I_1}'}(Y_{I_1}' \cdot B Y_{I_2}' = 0) - \frac{1}{q}| = |\P_{Y_{I_1}'}(Y_{I_1}' \cdot B' Y_{I_2}' = 0) - \frac{1}{q}|\le (C'/q)^{cm}.$$
 \end{proof}
Putting together by changing the constants, we obtain
\begin{prop}[Rank relations for symmetric matrices]\label{prop:rankevolution:sym} Assume that $\sqrt{\al} n \le m\le n$ and the symmetric matrix $M_n$ is as in Theorem \ref{thm:sym}. Then there exist positive constants $c,C'$ such that 
\begin{itemize}
\item Assume that $1\le k\le cm$, then
$$\Big|\P\big(\rk(M_{m+1}) \le \rk(M_m)+1 | \rk(M_m)=m-k \wedge \CE_m \wedge \CF_m)\big) -\frac{1}{q^{k}}\Big| \le \Big(\frac{C'}{q}\Big)^{cm}.$$
\item Assume that $0\le k\le cm$, then
$$\Big|\P\big(\rk(M_{m+1}) = \rk(M_m) | \rk(M_m)=m-k\wedge \CE_m \wedge \CF_m)\big) -\frac{1}{q^{k+1}}\Big| \le \Big(\frac{C'}{q}\Big)^{cm}.$$
\end{itemize}
\end{prop}
\begin{prop}[Rank relations for alternating matrices]\label{prop:rankevolution:alt} Assume that $\sqrt{\al} n \le m\le n$ and the alternating matrix $M_m$ is as in Theorem \ref{thm:alt}. Then there exist positive constants $c,C'$ such that the following holds for any $0\le k\le cm$ \footnote{Strictly speaking, Lemma \ref{lemma:linear} just gave $k\ge 1$, but for $k=0$ the bound automatically holds with probability one.}
$$\Big|\P\big(\rk(A_{m+1}) = \rk(A_m)  | \rk(A_m)=m-k\wedge \CE_m \wedge \CF_m)\big) -\frac{1}{q^{k}}\Big| \le   \Big(\frac{C'}{q}\Big)^{cm}.$$
\end{prop}

Now Theorems~\ref{thm:sym} and \ref{thm:alt} are almost in reach. We recall that in the independent case we started from a full rank matrix of size $n \times m_0$ (with $m_0=\lfloor (1-6\al) n \rfloor$) and expose the remaining columns one by one. Aside from the negligible events, we showed there that the rank statistics is asymptotically that of the uniform matrix of size $n-m_0$. The situation in the symmetric and alternating cases is different: instead of a column exposure process we have a corner exposure process, and because of this the situation is more complicated. To motivate the reader, let's say we use Lemma \ref{lemma:O:sym} to evolve from $m_0=\lfloor \sqrt{\al}n\rfloor$ with a matrix $M_{m_0}$ whose corank is at most $\eps m_0$ (but we will not use this information). Assume for now that at each step we append a new {\it uniform} vector to move from a square matrix of size $m$ to a square matrix of size $m+1$, what is the rank statistics of the final matrix? In this case we establish the following, which is Theorem \ref{thm:eventualequiv} restated here for convenience.

\begin{prop}\label{prop:eventualequiv}  Assume that $M_{m_0}$ is an arbitrary matrix of size $m_0$. Let $M_n(M_0)$ be the $n\times n$ random matrix with $M_0$ as its upper right hand corner and the remaining entries above the diagonal sampled uniformly and those below chosen to make the matrix alternating or symmetric. Then,
	\begin{align*}
		\|Q(M_n(M_0)) - Q_\bullet\|_{TV}\leq \frac{3^{n/2}}{q^{n/2-m_0}}.
	\end{align*}
	Where $Q_\bullet$ is either $Q_{sym,\infty}$,$Q_{alt,e}$ or $Q_{alt,o}$ depending on whether $M_n$ is symmetric or alternating and the parity of $n$. 
\end{prop}

We recall that the random matrix model above  is a special case of Theorem \ref{thm:sym}
 and Theorem \ref{thm:alt}, and hence inevitable. In order to prove Proposition \ref{prop:eventualequiv}, a key problem is that at any given step the corank isn't likely zero. In fact it is likely similar to the final distribution but not within the error tolerance we are aiming for. Our solution to this problem is to show that at {\it some} step of the corner exposure process with high probability the corank is zero. As a consequence, because at that step the corank matches with that of the uniform model at step zeroth, we can evolve as in the uniform model from there on. We will leave a linear length stretch to get our desired final distribution with our desired error bound.

\begin{lemma}\label{lemma:hitzero} Assume that $\eps > \sqrt{\al}$. Let $M_0$ be a symmetric or alternating matrix of corank $x_0  \le  m_0$. If we add the rows and columns according to the uniform model then 
	$$\P( \exists i \le \eps n, x_{i}=0) \geq 1 - \frac{3^{\eps n}}{q^{\eps n -m_0}}.$$
\end{lemma}

\begin{proof}(of Lemma \ref{lemma:hitzero}) In what follows we $x_m$ be the corank of the matrix $M_{m_0+m}$. 
	At each step there are at most three possibilities for the corank: it could increase by one, decrease by one or stay the same. Therefore there are at most $3^{\eps n}$ possible paths. At most because paths which have negative corank at some step are impossible and because the corank flatting is not possible in the alternating model. We compute the probability of the most likely path which has corank strictly greater than zero at every step then take a union bound.

\begin{claim}\label{claim:path} Among all sequences of coranks strictly greater than $0$ the one  with highest probability is the one which decreases to 1 and then alternates between 1 and 2. The probability of this path is less than,
	\begin{align*}
		\bigg(\frac{1}{q}\bigg)^{\eps n -m_0}.
	\end{align*}
\end{claim}
\begin{proof}(of Claim \ref{claim:path})
	We prove the claim for the symmetric rank relations and then describe the necessary modifications for the alternating model. We first recall the following transition probabilities for the uniform symmetric model
	$$\P(x_{m+1} =x_m+1) =\frac{1}{p^{x_m+1}} $$
and
$$\P(x_{m+1} =x_{m}-1) = 1- \frac{1}{p^{x_m}}$$
as well as
$$\P(x_{m+1}=x_m) = \frac{1}{p^{x_m}} - \frac{1}{p^{x_m+1}} = \frac{1}{p^{x_m}}(1-1/p) .$$	
	Denote by $+,-,0$ if the corank increases, decreases, and levels respectively, and by a string of those characters a sequence of those moves for our row-column appending process. We compare the probabilities of the following pairs of transitions at step $m$:
	$$\P(+\ 0) =  \frac{1}{q^{x_m+1}} \frac{1}{q^{x_m+1}}(1-1/q)<   \frac{1}{q^{x_m}}(1-1/q) \frac{1}{q^{x_m+1}} = \P(0\ +).$$
Also
$$\P(-\ 0) = (1- \frac{1}{q^{x_m}}) \frac{1}{q^{x_m-1}}(1-1/q) >   \frac{1}{q^{x_m}}(1-1/q) (1- \frac{1}{q^{x_m}}) = \P(0\ -)$$
and
$$\P(-\ +) = (1- \frac{1}{q^{x_m}}) \frac{1}{q^{x_m}} > \frac{1}{q^{x_m+1}}  (1- \frac{1}{q^{x_m+1}})=\P(+\ -).$$

So we have 
$$\P(-...-+...+0) \le \P(-...-+...0+) \le  \P(-...-0+...+).$$ 
Furthermore, 
$$\P(-...-+...+-) \le \P( -...-+...-+) \le \P(-...-+-+...+).$$ 

Repeating the above comparison we can then arrive at sequences of the form 
$$-...-0...0+-0...0+-0...0+-0...0+...+.$$
Observe we can replace the trailing $+...+$ by increasing the number of zeros, and we can move the zeros to the right side without changing the probability, so the maximum sequences might have the form
$$-...-+-+-...+-0...0.$$
Lastly, we have 
$$P(00) =  (\frac{1}{q^{x_m}}(1-1/q))^2 \le P(+\ -)=\frac{1}{q^{x_m+1}} (1- \frac{1}{q^{x_m}})$$ 
even when $x_m=1$. So the most likely sequence is $-...-+-+-...+-$ as claimed.  

	To bound the probability of this sequence observe there are at least $(\eps n -(m_0-x_0))/2$ pluses since the initial rank is at most $m_0-x_0$, that is the maximum length of the leading string of $-$'s. Each plus occurs with probability at most $1/q^2$ so we get the desired bound.

For the alternating model observe the corank either increases or decreases by one at each step. Staying the same isn't an option. The probability of the corank decreasing in the alternating model is equal to the sum of the probability of the rank staying the same and decreasing in the symmetric model. Therefore the alternating model has the same most probable corank sequence with the same bounding probability.
  \end{proof}
  From the claim we have,
	\begin{align*}
		\P( \forall i \le \eps n(x_i > 0)) \leq 3^{\eps n} q^{-(\eps n -m_0)},
	\end{align*}
	completing the proof.
\end{proof}

  \begin{proof}(of Proposition \ref{prop:eventualequiv})
	By the previous lemma with probability $1-3^{\eps n}/q^{\eps n -m_0}$ for some index $i\leq \eps n$, $Q(M_{m_0+i}) = 0$. The corank evolution only depends on the corank in the previous step so the distribution at the $n^{th}$ step is the same as the distribution of the corank of an $(n-i)\times (n-i)$ matrix. By \eqref{eqn:FG:sym} and \eqref{eqn:FG:alt} the total variation distance between this and $Q_\bullet$ is less than $3/q^{ n -\eps n - m_0}$. Choose $\eps n = n/2$ we obtain as claimed.
\end{proof}

\begin{proof}(of Theorem \ref{thm:sym} and \ref{thm:alt})
	Let $M_0$ be the upper left $m_0 \times m_0$ corner of our near uniform random matrix $M_n$. By our rank relations \ref{prop:rankevolution:sym} and \ref{prop:rankevolution:alt} at each of the remaining $n-m_0$ steps the evolution of the uniform and near uniform models differ by an error of at most $(C'/q)^{cm_0}$. Therefore summing over all possible indices of rank drops we have,
	\begin{align*}
		\|Q(M_n(M_0)) - Q(M_n)\|_{TV}\leq  2^n (\frac{C'}{q})^{cm_0}.
	\end{align*}
	By the previous Corollary we have,
	\begin{align*}
		\|Q(M_n(M_0) - Q_\bullet\|_{TV}\leq  \frac{3^{n/2}}{q^{n/2-m_0}}.
	\end{align*}
	Therefore by the triangle inequality we have,
	\begin{align*}
		\|Q(M_n)- Q_\bullet\|_{TV}\leq 2^n (\frac{C'}{q})^{cm_0} +\frac{3^{n/2}}{q^{n/2-m_0}}.
	\end{align*}
\end{proof}

\section{Proof of Theorem \ref{thm:GL:1} and Theorem \ref{thm:GL:2}}\label{section:GL}

Let $A_n\in \GL_n(\F_q)$ be chosen uniformly. Let $X_1,\dots,X_n$ denote the columns of $A_n$. We start with two simple results which will be important for both theorems: a uniform element of $\GL_n(\F_q)$ can be sampled using the column exposure process and a bound on the probability a rectangular submatrix of $A_n$ is full rank.

\begin{claim}\label{claim:GLn} Sample vectors $X_1,\dots, X_n$ one at a time as follows: for $0\le k\le n-1$ sample $X_{k+1}$ uniformly from $\F_q^n\bs \lang X_1,\dots, X_k\rang$. Let $A$ be the matrix with $X_i$ as columns. Then for each $B\in \GL_n(\F_q)$
$$\P(A=B) = \frac{1}{|\GL_n(\F_q)|}.$$
\end{claim}
\begin{proof} Let $\col_i$ be the $i$-th column vector of $B$. By definition $\P(X_{k+1} = \col_{k+1}(B)|X_1=\col_1(B),\dots, X_k= \col_k(B)) = \frac{1}{q^n - q^k}$, and so $\P(A=B) = \prod_{k=0}^{n-1} \frac{1}{q^n - q^k}.$
\end{proof}

\begin{lemma}\label{GL:submatrix:FR} For $1\le k \leq l \leq n$, let $X_1',\dots,X_k'$ be the vectors obtained from $X_1,\dots, X_k$ by restricting to the first (or last) $l$ coordinates. We have 
	$$\P(X_1',\dots, X_k' \mbox{ are linearly independent}) \geq 1-2/q^{l-k}.$$
	By symmetry for $X_1'',\dots,X_l''$ the restrictions of $X_1,\dots X_l$ to the first (or last) $k$ coordinates we have,
	$$\P(X_1'',\dots, X_l'' \mbox{ span $\F_q^{k}$}) \geq 1-2/q^{l-k}.$$
\end{lemma}
\begin{proof} We show that the event that the $X_i'$ are dependent is bounded by $2/q^{l-k}$. If the $X_i'$ are dependent there exists $t < k$ such that $X'_1,\dots,X'_t$ are linearly independent and $X'_{t+1}\in\langle X'_1,\dots,X'_t\rangle$. 
		
	We have that $X_{t+1}' \in \lang X_1',\dots, X_{t}' \rang$ if $X'_{t+1}$ belongs to the set $S_{t}$, where $S_{t}$ is the set whose projection onto the first $l$ coordinates is a subset of $\lang X_1',\dots, X_{t}' \rang$. If $X_1',\dots X_t'$ are linearly independent then $S_{t}$ has size $q^{t}\times q^{n-l} =q^{n+t-l}$. Therefore summing over the indices $t<k$ we obtain,
			\begin{align*}
				\P(X_1',\dots, X_k' \mbox{ are linearly dependent}) &\leq \frac{S_0}{|\overline{W}_{0}|} + \dots + \frac{S_{k-1}}{|\overline{W}_{k-1}|} \\ &\leq \frac{q^{n-l}}{q^n-1}  +\dots + \frac{q^{n+(k-1)-l}}{q^n-q^{k-1}} \\
								&\leq 2(q^{-l} + \dots + q^{k-l-1}) \\
										&\leq 2q^{k-l}.
			\end{align*}
	For the symmetric statement note that $X_1'',\dots,X_l''$ spanning $\F_q^k$ is equivalent to the first $k$ rows' restriction to the first $l$ coordinates being linearly independent.
\end{proof}

\subsection{Proof of Theorem \ref{thm:GL:1}} Let $M_n = A_n - I$. The method is very similar to the one in Section~\ref{section:iid}. Claim~\ref{claim:GLn} justifies our use of a column exposure process. From Lemma~\ref{GL:submatrix:FR} it follows $X_1-\Be_1,\dots, X_k-\Be_k$ are linearly independent for $k<(1/2 - c)n$ with probability $1-q^{-cn}$. By Proposition~\ref{prop:GL:evolution} the normal vectors to $W_k$ are likely unstructured. The proof concludes similarly to the proof of Theorem~\ref{thm:main:1}.

From Lemma \ref{GL:submatrix:FR} it follows that $W_k(M_n)$ has dimension $k$ with probability at least $1-q^{-cn}$ for $k\le 1/2(1 -c)n$ so we may start analyzing the rank evolution process there.  For $k\geq (1/2 + \eps) n$, let $\CF_{GL,k}$ be the event that $X_1',\dots,X_k'$ span $\F_q^{n-k}$ where $X_1',\dots,X_k'$ are the restrictions of $X_1,\dots,X_k'$ to the last $n-k$ coordinates. For other $k$ let $\CF_{GL,k}$ be empty. By Lemma \ref{GL:submatrix:FR}, $\P(\CF_{GL,k})>1-q^{n-2k}$. We now study how the rank evolves after exposing $X=X_{k+1}$. 

\begin{proposition}\label{prop:GL:evolution} Let $k\ge (1/2-\eps)n$. For any nonzero normal vector $\Bw$ of $\lang X_1-\Be_1,\dots, X_k-\Be_k\rang$ we have
	$$\Big|\P((X_{k+1}-\Be_{k+1}) \cdot \Bw=0 | \CF_{GL,k})-1/q\Big|\leq q^{(-1/2+\eps)n}.$$
\end{proposition}
\begin{proof} Let $Y_1,\dots,Y_{n-k}$ be an orthonormal basis for $\overline{W_k}$. We can view $X_{k+1}$ as a random vector $\sum_{i=1}^k \beta_i X_i + \sum_{i=1}^{n-k} \al_i Y_i$ where $(\al_1,\dots,\al_{n-k},\beta_1,\dots,\beta_k)$ is chosen uniformly from the set of vectors where at least one of $\al_i$ is nonzero. In particular each $\beta_i$ is independent of $\alpha_1,\dots,\alpha_{n-k},\beta_1,\dots,\hat{\beta_i},\dots,\beta_k$.

	{\bf Case 1:}  $k\ge (1/2+\eps)n$.
	
	Because we conditioned on $X_1',\dots, X_k'$ spanning $\F_q^{n-k}$, $\Bw$ is not orthogonal to one of $\Be_1,\dots, \Be_k$. Without loss of generality assume that $\Bw$ is not orthogonal to $\Be_1$. Now as $\Bw$ is orthogonal to $X_1 - \Be_1,\dots, X_k - \Be_k$, we have 
	\begin{align*}\P_{\beta_1}((X_{k+1} -\Be_{k+1}) \cdot \Bw =0) &= \P_{\beta_1}\Big(  (\sum_{i=1}^k \beta_i X_{i} + \sum_{i=1}^{n-k} \al_i Y_i - \Be_{k+1}) \cdot \Bw = 0\Big) \\ &= \P_{\beta_1}\Big(  \beta_1 X_1 \cdot \Bw = a\Big)\\ &= \P_{\beta_1}(\beta_1 \Be_1 \cdot \Bw=a)\\ &=1/q.\end{align*}
		Note $a = -\sum_{i=2}^k \beta_i X_i - \sum_{i=1}^{n-k} + \Be_{k+1}$ is independent of $\beta_1$ and $\Be_1 \cdot \Bw = X_1\cdot\Bw\neq 0$.

	{\bf Case 2:} $(1/2-\eps) n \le k\le (1/2+\eps)n$.
	
	If $\Bw$ is not orthogonal to some $\Be_i, 1\le i \le k$ we can use the argument from Case 1. So assume $\Be_i\cdot\Bw=0$ for $i\leq k$. Because $\Bw\cdot (X_i - \Be_i) = 0$ for $i\leq k$ this also implies $X_i\cdot \Bw = 0$ for $i\leq k$. Write $\Bw = (w_1,\dots,w_n) = (0,\dots, 0, w_{k+1},\dots,w_n)$ with respect to the basis $\{X_1,\dots,X_k,Y_1,\dots,Y_{n-k}\}$. The tuple $(\al_1,\dots,\al_{n-k})$ is chosen uniformly from $\F_q^{n-k}\bs \{\mathbf{0}\}$, and $(w_{k+1},\dots,w_n)$ is a fixed vector in $\F_q^{n-k}\bs \{\mathbf{0}\}$. If $X$ was chosen uniformly from $\F_q^{n-k}$ then $\P((X-\Be_{k+1})\cdot \Bw = 0) = 1/q$. Accounting for whether $X=0$ would or would not have given an equality we get,
	\begin{align*}
		\Big|\P\Big((X-\Be_{k+1})\cdot \Bw = 0\Big) - 1/q\Big| &\leq \max\Big\{\frac{q^{n-k-1}}{q^{n-k}-1} - 1/q,1/q-\frac{q^{n-k-1}-1}{q^{n-k}-1}\Big\} \\
		&\leq \frac{1}{q^{n-k-1}}.
	\end{align*}
\end{proof}
\begin{proof}(of Theorem \ref{thm:GL:1}) Combining Proposition \ref{prop:GL:evolution} and Lemma \ref{unconcimpliesuniform} we then obtain that for any $k\ge (1/2-\eps)n$,  with $X_{k+1}$ being chosen uniformly from $\F_q^n \bs W_k(A_n)$ we have
	$$\Big|\P_{X_{k+1}}(X_{k+1} -\Be_{k+1}\in W_k(M_n)| \codim(W_k(M_n))=l) - \frac{1}{q^{l}}\Big|\leq q^{(-1/2+\eps)n}.$$
	Using this result together with the fact that $\codim(W_{\lfloor (1/2 - \eps)n \rfloor}(M_k)) = 0$ with probability $1-p^{-\eps n}$, we can complete the proof of Theorem \ref{thm:GL:1} the same way we concluded Theorem \ref{thm:main:1}. Namely there are $2^n$ ways the rank could evolve during the column exposure process. Choosing $\eps=1/4$ each occurs with the probability one would predict from the uniform model except for an error bounded by,
	\begin{align*}
		2q^{-n/2} + (1-1/4)nq^{-n/4}.
	\end{align*}
	Where the errors come from the probability $W_{\lfloor n/4 \rfloor}(M_n)$ is not full rank and the error we accumulate in the remaining $\lfloor (1-1/4)n \rfloor$ steps. So for $q>2^4$ we can sum our errors and get our desired bound.
\end{proof}
\subsection{Proof of Theorem~\ref{thm:GL:2}}

In this subsection, let $M_{n'}$ denote the upper left $n'\times n'$ corner of $A_n$ where $n'\leq (1-\eps)n$ and let $X_1',\dots,X_{n'}'$ denote its columns. 

Just as in the independent and perturbed models the goal is to apply Proposition~\ref{unconcimpliesuniform} to determine the transition probabilities. With that in mind we show the following.

\begin{proposition}\label{prop:corner:evolution}
	For all nonzero normal vectors $\Bw$ to $W_k(M_{n'})$ we have 
	$$|\P(\Bw\cdot X_{k+1}' = 0) - 1/q| < \frac{1}{q^{n-k-1}}.$$
\end{proposition}
\begin{proof}
	Let $Y_1,\dots,Y_{n-k}$ be a basis for $W_k(A_n)$. As in Proposition~\ref{prop:GL:evolution} $X_{k+1} = \sum_{i=1}^k \beta_i X_i + \sum_{i=1}^{n-k} \al_i Y_i$ where $(\al_1,\dots,\al_{n-k},\beta_1,\dots,\beta_k)$ is chosen uniformly from the set of vectors where at least one of $\al_i$ is nonzero. Now consider the restriction of this equation to the first $n'$ entries,
	\begin{align*}
		X_{k+1}' = \sum_{i=1}^k \beta_i X_i' + \sum_{i=1}^{n-k} \al_i Y_i'.
	\end{align*}
	Taking the dot product with $\Bw$ we get,
	\begin{align*}
		\Bw\cdot X_{k+1}' &= \sum_{i=1}^k \beta_i X_i'\cdot\Bw + \sum_{i=1}^{n-k} \al_i Y_i'\cdot\Bw \\
		&= \sum_{i=1}^{n-k} \al_i Y_i'\cdot\Bw
	\end{align*}
	
	At least one $Y_i$ must have $Y_i\cdot \Bw\neq 0$. Therefore as in Proposition~\ref{prop:GL:evolution} we get,
	\begin{align*}
		\Big|\P\Big((X_{k+1}'\cdot \Bw = 0\Big) - 1/q\Big| &\leq \max\Big\{\frac{q^{n-k-1}}{q^{n-k}-1} - 1/q,1/q-\frac{q^{n-k-1}-1}{q^{n-k}-1}\Big\} \\
		&\leq \frac{1}{q^{n-k-1}}.
	\end{align*}
\end{proof}
\begin{proof}(of Theorem \ref{thm:GL:2})
	Combining Proposition~\ref{prop:corner:evolution} and Proposition~\ref{unconcimpliesuniform} we obtain,
	\begin{align*}
		|\P(X_{k+1}'\in W_k(M_{n'})|\codim(W_k(M_{n'})) = d) - 1/p^d| &\leq \frac{1}{q^{n-k-1}} \\
		&\leq q^{-\eps n}.
	\end{align*}
	As a consequence, by directly comparing with the rank evolution of the uniform model over size $n'$ (via \eqref{eqn:uni:iid})  and by taking union bound, we obtain a total variation distance of $2^{n'}q^{-\eps n}$ for the rank distribution of this model and the uniform $n'\times n'$ model. This bound can be improved if $n'$ is large, for instance when $n' \ge \eps n$. Indeed, in this case, by Lemma~\ref{GL:submatrix:FR} $W_{n' -\lfloor \eps n \rfloor}(M_{n'})$ has full dimension $n' -\lfloor \eps n \rfloor$ with probability at least $1-2q^{-\lfloor \eps n \rfloor}$. Starting at this step leaves $2^{\lfloor \eps n \rfloor}$ ways for the rank to evolve. As such we obtain a total variation distance of $2q^{-\lfloor \eps n \rfloor} + 2^{\lfloor \eps n \rfloor}q^{-\eps n}$ between the rank distribution of this model and the uniform $n'\times n'$ model. Finally, using Fulman and Goldstein's result we add another term of $\frac{3}{q^{n'}}$ for the total variation distance to the limiting distribution $Q_\infty$.
\end{proof}

\end{document}